\documentclass[11pt,reqno,oneside]{amsproc}
\title[Lower bounds on blowup rate]{
On the blowup rate of vorticity for the Euler equations in a bounded domain}

\author[B.~Ingimarson]{Benjamin Ingimarson}
\address{Department of Mathematics, University of Southern California, Los Angeles, CA 90089}
\email{ingimars@usc.edu}

\author[I.~Kukavica]{Igor Kukavica}
\address{Department of Mathematics, University of Southern California, Los Angeles, CA 90089}
\email{kukavica@usc.edu}

  \chardef\forshowkeys=0
  \chardef\showllabel=0
  \chardef\refcheck=0
  \chardef\sketches=0
  \chardef\figures=1

\usepackage{enumitem,mathtools}
\usepackage{datetime}
\usepackage{fancyhdr}
\usepackage{comment}
\usepackage{listings}

\allowdisplaybreaks
\ifnum\forshowkeys=1

  \usepackage[notref,notcite,color]{showkeys}
\fi
\usepackage[margin=1in]{geometry}
\usepackage{amsmath, amsthm, amssymb}
\usepackage{times}
\usepackage{graphicx}
\usepackage[usenames,dvipsnames,svgnames,table]{xcolor}
\usepackage{marginnote}
\usepackage[unicode,breaklinks=true,colorlinks=true,linkcolor=blue,urlcolor=blue,citecolor=blue]{hyperref}
\usepackage[most]{tcolorbox}
\usepackage{tikz}

\ifnum\refcheck=1
  \usepackage{refcheck}
\fi

\usepackage{import}
\usepackage{xifthen}
\usepackage{pdfpages}
\usepackage{transparent}

\begin{document}
\def\YY{X}
\def\OO{\mathcal O}
\def\SS{\mathbb S}
\def\CC{\mathbb C}
\def\RR{\mathbb R}
\def\TT{\mathbb T}
\def\ZZ{\mathbb Z}
\def\HH{\mathbb H}
\def\RSZ{\mathcal R}
\def\LL{\mathcal L}
\def\SL{\LL^1}
\def\ZL{\LL^\infty}
\def\GG{\mathcal G}
\def\tt{\langle t\rangle}
\def\erf{\mathrm{Erf}}
\def\mgt#1{\textcolor{magenta}{#1}}
\def\ff{\rho}
\def\gg{G}
\def\sqrtnu{\sqrt{\nu}}
\def\ww{w}
\def\ft#1{#1_\xi}
\def\les{\lesssim}
\def\lec{\lesssim}
\def\ges{\gtrsim}
\def\gec{\gtrsim}
\renewcommand*{\Re}{\ensuremath{\mathrm{{\mathbb R}e\,}}}
\renewcommand*{\Im}{\ensuremath{\mathrm{{\mathbb I}m\,}}}
\ifnum\showllabel=1
 \def\llabel#1{\marginnote{\color{lightgray}\rm\small(#1)}[-0.0cm]\notag}
\else
 \def\llabel#1{\notag}
\fi
\newcommand{\norm}[1]{\left\|#1\right\|}
\newcommand{\nnorm}[1]{\lVert #1\rVert}
\newcommand{\abs}[1]{\left|#1\right|}
\newcommand{\NORM}[1]{|\!|\!| #1|\!|\!|}

\newtheorem{Theorem}{Theorem}[section]
\newtheorem{Corollary}[Theorem]{Corollary}
\newtheorem{Proposition}[Theorem]{Proposition}
\newtheorem{Lemma}[Theorem]{Lemma}
\newtheorem{Remark}[Theorem]{Remark}
\newtheorem{definition}{Definition}[section]

\def\theequation{\thesection.\arabic{equation}}
\numberwithin{equation}{section}
\definecolor{mygray}{rgb}{.6,.6,.6}
\definecolor{myblue}{rgb}{9, 0, 1}
\definecolor{colorforkeys}{rgb}{1.0,0.0,0.0}
\newlength\mytemplen
\newsavebox\mytempbox
\makeatletter
\newcommand\mybluebox{%
    \@ifnextchar[
       {\@mybluebox}%
       {\@mybluebox[0pt]}}
\def\@mybluebox[#1]{%
    \@ifnextchar[
       {\@@mybluebox[#1]}%
       {\@@mybluebox[#1][0pt]}}
\def\@@mybluebox[#1][#2]#3{
    \sbox\mytempbox{#3}%
    \mytemplen\ht\mytempbox
    \advance\mytemplen #1\relax
    \ht\mytempbox\mytemplen
    \mytemplen\dp\mytempbox
    \advance\mytemplen #2\relax
    \dp\mytempbox\mytemplen
    \colorbox{myblue}{\hspace{1em}\usebox{\mytempbox}\hspace{1em}}}
\makeatother

\def\pv{\text{p.v.}}
\def\bnew{\color{red}}
\def\enew{\color{black}}
\def\bold{\color{blue}}
\def\eold{\color{black}}
\def\rr{r}
\def\weaks{\text{\,\,\,\,\,\,weakly-* in }}
\def\inn{\text{\,\,\,\,\,\,in }}
\def\cof{\mathop{\rm cof\,}\nolimits}
\def\Dn{\frac{\partial}{\partial N}}
\def\Dnn#1{\frac{\partial #1}{\partial N}}
\def\tdb{\tilde{b}}
\def\tda{b}
\def\qqq{u}
\def\lat{\Delta_2}
\def\biglinem{\vskip0.5truecm\par==========================\par\vskip0.5truecm}
\def\inon#1{\hbox{\ \ \ \ \ \ \ }\hbox{#1}}                
\def\onon#1{\inon{on~$#1$}}
\def\inin#1{\inon{in~$#1$}}
\def\FF{F}
\def\andand{\text{\indeq and\indeq}}
\def\ww{w(y)}
\def\ll{{\color{red}\ell}}
\def\ee{\mathrm{e}}
\def\startnewsection#1#2{ \section{#1}\label{#2}\setcounter{equation}{0}}   
\def\nnewpage{ }
\def\sgn{\mathop{\rm sgn\,}\nolimits}    
\def\Tr{\mathop{\rm Tr}\nolimits}    
\def\div{\mathop{\rm div}\nolimits}
\def\curl{\mathop{\rm curl}\nolimits}
\def\dist{\mathop{\rm dist}\nolimits}
\def\id{\mathop{\rm id}\nolimits}
\def\supp{\mathop{\rm supp}\nolimits}
\def\indeq{\quad{}}           
\def\period{.}                       
\def\semicolon{\,;}                  
\def\cmi#1{\text{~{{\coli IK: \underline{#1}}}~}}
\def\coli{\color{colorigor}}
\definecolor{colorigor}{rgb}{.5, 0.2, 0.8}
\def\colr{\color{red}}
\def\colrr{\color{black}}
\def\colb{\color{black}}
\def\coly{\color{lightgray}}
\definecolor{colorgggg}{rgb}{0.1,0.5,0.3}
\definecolor{colorllll}{rgb}{0.0,0.7,0.0}
\definecolor{colorhhhh}{rgb}{0.3,0.75,0.4}
\definecolor{colorpppp}{rgb}{0.7,0.0,0.2}
\definecolor{coloroooo}{rgb}{0.45,0.0,0.0}
\definecolor{colorqqqq}{rgb}{0.1,0.7,0}
\def\colg{\color{colorgggg}}
\def\collg{\color{colorllll}}
\def\cole{\color{coloroooo}}
\def\coleo{\color{colorpppp}}
\def\colu{\color{blue}}
\def\colc{\color{colorhhhh}}
\def\colW{\colb}   
\definecolor{coloraaaa}{rgb}{0.6,0.6,0.6}
\def\colw{\color{coloraaaa}}
\def\comma{ {\rm ,\qquad{}} }            
\def\commaone{ {\rm ,\quad{}} }          
\def\les{\lesssim}
\def\nts#1{{\color{red}\hbox{\bf ~#1~}}} 
\def\ntsf#1{\footnote{\color{colorgggg}\hbox{#1}}} 
\def\blackdot{{\color{red}{\hskip-.0truecm\rule[-1mm]{4mm}{4mm}\hskip.2truecm}}\hskip-.3truecm}
\def\bluedot{{\color{blue}{\hskip-.0truecm\rule[-1mm]{4mm}{4mm}\hskip.2truecm}}\hskip-.3truecm}
\def\purpledot{{\color{colorpppp}{\hskip-.0truecm\rule[-1mm]{4mm}{4mm}\hskip.2truecm}}\hskip-.3truecm}
\def\greendot{{\color{colorgggg}{\hskip-.0truecm\rule[-1mm]{4mm}{4mm}\hskip.2truecm}}\hskip-.3truecm}
\def\cyandot{{\color{cyan}{\hskip-.0truecm\rule[-1mm]{4mm}{4mm}\hskip.2truecm}}\hskip-.3truecm}
\def\reddot{{\color{red}{\hskip-.0truecm\rule[-1mm]{4mm}{4mm}\hskip.2truecm}}\hskip-.3truecm}
\def\tdot{{\color{green}{\hskip-.0truecm\rule[-.5mm]{3mm}{3mm}\hskip.2truecm}}\hskip-.1truecm}
\def\gdot{\greendot}
\def\bdot{\bluedot}
\def\ydot{\cyandot}
\def\rdot{\cyandot}
\def\fractext#1#2{{#1}/{#2}}
\def\ii{\hat\imath}
\def\fei#1{\textcolor{blue}{#1}}
\def\vlad#1{\textcolor{cyan}{#1}}
\def\igor#1{\text{{\textcolor{colorqqqq}{#1}}}}
\def\igorf#1{\footnote{\text{{\textcolor{colorqqqq}{#1}}}}}
\def\Omf{\Omega_{\text f}}
\def\Ome{\Omega_{\text e}}
\def\Omb{\Omega_{\text b}}
\def\Gaf{\Gamma_{\text f}}
\def\Gae{\Gamma_{\text e}}
\def\Gab{\Gamma_{\text b}}
\def\Gac{\Gamma_{\text c}}
\def\Nf{N^{\text f}}
\def\Ne{N^{\text e}}

\newcommand{\p}{\partial}
\renewcommand{\d}{\mathrm{d}}
\newcommand{\UE}{U^{\rm E}}
\newcommand{\PE}{P^{\rm E}}
\newcommand{\KP}{K_{\rm P}}
\newcommand{\uNS}{u^{\rm NS}}
\newcommand{\vNS}{v^{\rm NS}}
\newcommand{\pNS}{p^{\rm NS}}
\newcommand{\omegaNS}{\omega^{\rm NS}}
\newcommand{\uE}{u^{\rm E}}
\newcommand{\vE}{v^{\rm E}}
\newcommand{\pE}{p^{\rm E}}
\newcommand{\omegaE}{\omega^{\rm E}}
\newcommand{\ua}{u_{\rm   a}}
\newcommand{\va}{v_{\rm   a}}
\newcommand{\omegaa}{\omega_{\rm   a}}
\newcommand{\ue}{u_{\rm   e}}
\newcommand{\ve}{v_{\rm   e}}
\newcommand{\omegae}{\omega_{\rm e}}
\newcommand{\omegaeic}{\omega_{{\rm e}0}}
\newcommand{\ueic}{u_{{\rm   e}0}}
\newcommand{\veic}{v_{{\rm   e}0}}
\newcommand{\up}{u^{\rm P}}
\newcommand{\vp}{v^{\rm P}}
\newcommand{\tup}{{\tilde u}^{\rm P}}
\newcommand{\bvp}{{\bar v}^{\rm P}}
\newcommand{\omegap}{\omega^{\rm P}}
\newcommand{\tomegap}{\tilde \omega^{\rm P}}
\newcommand{\eps}{\varepsilon}  
\newcommand{\eqnb}{\begin{equation}}
\newcommand{\eqne}{\end{equation}}
  
\renewcommand{\up}{u^{\rm P}}
\renewcommand{\vp}{v^{\rm P}}
\renewcommand{\omegap}{\Omega^{\rm P}}
\renewcommand{\tomegap}{\omega^{\rm P}}

\begin{abstract}
Given that a solution to the 3D incompressible Euler equations on a \emph{bounded domain} blows up at a time $T_\ast$ and that $T_\ast$ is the first such time, we provide pointwise-in-time lower bounds on $\|D^k\omega\|_{L^\infty(\Omega)}$ for $k \geq 1$. 
We also show that the Gronwall-type inequality satisfied by $\|\omega(t)\|_{L^\infty}$, in the cases that $\Omega = \mathbb{R}^3$, $\mathbb{T}^3$, or a bounded domain, exhibits wildly oscillating solutions. 
\end{abstract}

\maketitle

\setcounter{tocdepth}{2} 
\section{Introduction}\label{sec00}
Consider the 3D incompressible Euler equations,
  \begin{align}
   \begin{split}
    u_t 
    + u\cdot \nabla u
    + \nabla p
    &= 0
    ,
    \\
    \div u 
    &= 0
    ,
   \end{split}
   \label{EQ00}
  \end{align}
in $\Omega \times [0,T]$ with $u \cdot n = 0$ on $\partial \Omega$, where $\Omega$ is a smooth bounded domain.
It is well known that for an initial velocity field $u_0 \in H^r$ with $r > 2.5$, where $\|u_0\|_{H^r} \leq M$ for some $M > 0$, there exists a time $T = T(M) > 0$ such that~\eqref{EQ00} has a solution in the class,
  \begin{equation}
   u
   \in
   C([0,T]; H^r(\Omega)) \cap C^1([0,T];H^{r-1}(\Omega))
   .
   \label{EQ01}
  \end{equation}
It was demonstrated by Ferrari~\cite{F} that if such a solution leaves the class~\eqref{EQ01} at a time $T_\ast$ and that $T_\ast$ is the first such time, then 
  \begin{equation}
   \int_0^{T_\ast} \|\omega(t)\|_{L^\infty(\Omega)} 
   = \infty
   ,
   \quad\quad
   \text{and}
   \quad\quad
   \limsup_{t \to T_\ast} \|\omega(t) \|_{L^\infty(\Omega)} 
   = \infty
   .
   \label{EQ02}
  \end{equation}
The criterion outlined in~\eqref{EQ01} and~\eqref{EQ02} was first shown in the seminal work by Beale, Kato, and Majda~\cite{BKM} when $\Omega = \mathbb{R}^3$ or $\mathbb{T}^3$;
see also \cite{MB,MP} for expositions.
Other than the aforementioned case of a bounded domain, many variants of the criterion have been proven, see~\cite{A,C,CFM,P,KOT,KT}.
For some time-dependent lower bounds on the rate of growth of various norms, see~\cite{C,CIV,CP}, while
for some recent blow-up results of the Euler equations, see~\cite{CMZ,E,JMO,S}.

In a previous work~\cite{IK}, we established lower bounds on the blow-up rate of $\int_0^t \|\omega(s)\|_{L^\infty(\Omega)}$ when $\Omega = \mathbb{R}^3$, $\mathbb{T}^3$, or a smooth bounded domain.
We also provided pointwise-in-time lower bounds on $\|D^k\omega\|_{L^\infty(\Omega)}$ when $\Omega = \mathbb{R}^3$ or~$\mathbb{T}^3$.
The proofs of the latter results depend on the absence of a boundary.  

In this paper, we derive pointwise-in-time lower bounds on $\|D^k\omega\|_{L^\infty(\Omega)}$ when $\Omega$ is a smooth bounded domain.
We also return to the Gronwall-type inequality for $\|\omega\|_{L^\infty(\Omega)}$, showing that it admits solutions which oscillate wildly up to the blow-up time.
In particular, this implies that no nontrivial pointwise-in-time lower bound can be deduced from the inequality alone. 

In Section~\ref{sec01}, we give the main results of the paper. 

In Section~\ref{sec02}, we prove Theorems~\ref{T01} and~\ref{T02}. 
Theorem~\ref{T01} gives pointwise-in-time lower bounds on $\|D\omega\|_{L^\infty}$ with a log-correction. 
In comparison with~\cite{IK}, where we relate $\|D^{1/2}\omega\|_{L^p}$ to $\|D\omega\|_{L^\infty}$ through the Gagliardo-Nirenberg interpolation inequality, the analysis for $\|D\omega\|_{L^\infty}$ is complicated by the presence of a boundary; 
there is no clear evolution estimate which closes for~$\|D^{1/2}\omega\|_{L^p}$. 
The proof instead proceeds by an optimization argument, 
for which the dependence on $p$ for the constants in several well-known inequalities must be tracked.
Lower bounds on the $L^p$ norm are then derived from the vorticity equation which then pass to lower bounds, dependent on $p$, for the $L^\infty $ norm. 
The optimization then follows, treating $p$ as a function of $t$, to give a lower bound at the critical exponent $7/5$ with a log-correction.
We also prove Theorem~\ref{T02}.
The proof is similar to that of~\cite[Theorem~2.6]{IK}, except for the appearance of lower-order terms originating from the presence of a boundary. 

In Section~\ref{sec03}, we revisit the differential inequality satisfied by $\|\omega\|_{L^\infty}$, see~\cite[Theorem~2.1]{IK}. 
Under the usual BKM hypotheses for $r =3$, we showed that
  \begin{equation}
   \dot{x} 
   \leq
   Cx 
   + Cx^2 \exp\left(C \int_0^t x \right)
   ,
   \label{EQ02a}
  \end{equation}
where $x(t) = \|\omega(t)\|_{L^\infty}$.
In fact, we showed in \cite{IK} that
\eqref{EQ02a} implies
  \begin{equation}
   \limsup_{t\to T_\ast^-} (T_\ast -t) \log \left( \frac{1}{T_\ast -t}\right) \|\omega(t)\|_{L^\infty} \geq \frac{1}{C}
   ,
   \label{EQ02b}
  \end{equation}
where the constant $C$ depends on~$\|u_0\|_{H^3}$.
The inequality \eqref{EQ02b} holds when
$\Omega$ equals
$\mathbb{R}^{3}$,
$\mathbb{T}^{3}$,
or a smooth bounded domain~$\Omega$.
The inequality~\eqref{EQ02a} is thus of considerable interest,
and a natural question is whether or not the lower bound on the supremum in~\eqref{EQ02b} can be improved to a pointwise-in-time lower bound on~$\|\omega(t)\|_{L^\infty}$. 
Usually, under the assumption of finite-time blow-up, Gronwall-type inequalities imply
non-oscillatory behavior near the blow-up time,
i.e.,
that the limit of the quantity is~$\infty$.
For example, the inequalities $\dot{x} \leq x^2$ or, perhaps $\dot{x} \leq x^2 e^x$, both imply non-oscillatory blow-up and, hence, pointwise-in-time lower bounds on the blow-up rate. 
However, as is evident in~\eqref{EQ02a}, this analysis is greatly complicated by the non-local term in the exponential. 
In Theorem~\ref{T03} we show that such pointwise-in-time lower bounds are impossible to deduce from~\eqref{EQ02a} alone.
Indeed, we construct a $C^1$ function of time which satisfies~\eqref{EQ02a} and blows up at $T_\ast$ yet oscillates wildly up to the blow-up time $T_\ast$; in particular, there is a sequence of times $t_n \to T_\ast^-$ such that $x(t_n) \to 0$. 
As was shown in~\cite{IK}, the inequality~\eqref{EQ02a} can be recast as
  \begin{equation}
   \frac{d}{d\tau} X(\tau) 
   \les
   1 + X(\tau)e^\tau
   ,
   \label{EQ02c}
  \end{equation}
where $\tau(t) = \int_0^t x(s) \, ds$ denotes the accumulation of~$x$. 
In the proof of Theorem~\ref{T03}, we reverse this logic by showing any solution to~\eqref{EQ02c} gives a solution to~\eqref{EQ02a}, under an analogous reparameterization. 
The desired oscillating solution is then constructed for the more agreeable inequality~\eqref{EQ02c}.

We also note that, when $\Omega = \mathbb{R}^3$
and in the case when the initial velocity belongs only to
$C^{1,\alpha}$, where $\alpha\in(0,1)$,
the quantity $x(t) = \|\omega(t)\|_{L^\infty} + [\omega(t)]_{C^{0,\alpha}}$ satisfies the bound $x(t) \les \exp \exp \left( \int_0^t x \right)$, from which we can deduce the
lower bound \eqref{EQ02b}
for $\Vert \omega(t)\Vert_{L^{\infty}(\Omega)}$
and a bound consistent with those in \cite{IK}
for
$\Vert\omega(t)\Vert_{C^{0,\alpha}(\Omega)}$.

\section{Preliminaries and main results}\label{sec01}
For the rest of the paper, we  denote $L^p = L^p(\Omega)$, where $\Omega$ is a smooth bounded domain. 
The first theorem gives pointwise-in-time lower bounds on the blow-up rate of~$\|D\omega\|_{L^\infty}$. 
\begin{Theorem}
\label{T01}
Let $u$ be a solution to the Euler equations~\eqref{EQ00} for a smooth bounded domain $\Omega \subset \mathbb{R}^3$ in the regularity class~\eqref{EQ01} with $r>7 /2$.
Suppose there is a time $T_\ast > 0$ such that the solution cannot be continued in this class to $T = T_\ast$ and that $T_\ast$ is the first such time.
Then, there exists a constant $\rho > 0$ such that
  \begin{equation}
   \|D\omega(t)\|_{L^\infty}
   \geq 
   \frac{1}{C(T_\ast -t)^{7/5} \log^{\,\rho} \left( \frac{1}{T_\ast -t}\right)}
   ,
   \quad\quad
   T_\ast -t < \frac{1}{C}
   ,
   \label{EQ03}
  \end{equation}
where $C$ is a constant depending on~$\|u_0\|_{L^2}$.
\end{Theorem}
For example, the proof of Theorem~\ref{T01} affirms~\eqref{EQ03} for $\rho = 24/5$.

The second theorem gives pointwise-in-time lower bounds on the blow-up rate of $\|D^k \omega\|_{L^\infty}$ for $k \geq 2$.
\begin{Theorem}
\label{T02}
Assume the hypothesis of Theorem~\ref{T01} for $r > 5/2 + m$ where $m \in \{3,4,\dots\}$.
Then, we have 
  \begin{equation}
   \|D^k\omega(t)\|_{L^\infty}
   \geq
   \frac{1}{C(T_\ast -t)^{2k/5 + 1}}
   ,
   \quad\quad
   T_\ast -t < \frac{1}{C}
   \llabel{EQ04}
  \end{equation}
for $k = 2,\dots, m$, where $C$ depends on~$\|u_0\|_{L^2}$. 
\end{Theorem}
The final theorem is a statement about the Gronwall-type inequality satisfied by $\|\omega(t)\|_{L^\infty}$, given in~\eqref{EQ02a}.
Without loss of generality, we consider $C=1$. 
The inequality then reads
  \begin{equation}
   \dot{x}
   \leq
   x 
   + x^2 \exp\left(\int_0^t x\, ds\right)
   ,
   \quad\quad
   t \in [0,T^\ast)
   .
   \label{EQ05}
  \end{equation}
The final theorem shows that~\eqref{EQ05} admits wildly oscillating $C^1$ solutions which blow up at~$T_\ast$.

\begin{Theorem}
\label{T03}
For any $M > 0$, there exists a time $T_\ast = T_\ast(M)$ and a $C^1$ function $x \colon [0,T_\ast) \to (0,\infty)$ such that the following hold:
  \begin{itemize}
   \item[(i)] $x(0) = M$,
   \item[(ii)] $x$ satisfies~\eqref{EQ05}
   \item[(iii)] $\limsup_{t \to T_\ast^-} x(t) = \infty$, and
   \item[(iv)] there exists a sequence $t_n \to T_\ast^-$ such that $x(t_n) = \frac{M}{n+1}$.
  \end{itemize}
\end{Theorem}

Note that (iv) implies $\liminf_{t\to T_{*}^{-}} x(t)=0$.

\section{Lower bounds on $\|D^k\omega\|_{L^\infty}$ for $k = 1,2,\dots$}\label{sec02}
First, we prove Theorem~\ref{T01}. 
In the proof, it is important to track the dependence on $p$ for various inequalities involving the norm~$\|\cdot \|_{W^{k,p}}$. 
We start with the following lemma which gives the $p$-dependence in a Gagliardo-Nirenberg inequality as well as the div-curl lemma found in~\cite[Lemma~5]{BB}.
For the following lemma and the proof of Theorem~\ref{T01}, we write $O_p(1)$ to mean boundedness as $p \to \infty$.
We also write ``$C$'' or ``$\les$'' to refer to dependence on a constant which is~$O_p(1)$. 
\begin{Lemma}
\label{L04}
Assume the hypothesis of Theorem~\ref{T01}. For $p > 4$, we have 
  \begin{equation}
   \|Du\|_{L^\infty}
   \leq
   C\|u\|_{L^2}^{1-\alpha} \|D^2 u\|_{L^p}^\alpha
   + C \|u\|_{L^2}
   ,
   \quad\quad
   \text{where }
   \alpha = \frac{5}{7 - 6/p}
   .
   \label{EQ07}
  \end{equation}
Moreover, 
  \begin{equation}
   \|D^2 u\|_{L^p}
   \leq 
   Cp^2 \|D\omega\|_{L^p}
   + C p^{2\gamma} \|u\|_{L^2}
   ,
   \quad\quad
   \text{where } 
   \gamma = \frac{7p-3}{2p+3}
   .
   \label{EQ08}
  \end{equation}
\end{Lemma}
Note that, for~\eqref{EQ08}, $u$ must be divergence free with $u\cdot n = 0$ on~$\partial \Omega$.
\begin{proof}
To prove~\eqref{EQ07}, we start with the case $\Omega = \mathbb{R}^3$. 
We denote by $\Delta_j$ the non-homogeneous dyadic block operator  with $j \geq -1$, see~\cite[Section~2.2]{BCD} for details.
Writing $Du = \Delta_{-1} Du + (\id - \Delta_{-1}) Du$, we have, by Bernstein's inequality,
  \begin{equation}
   \|\Delta_{-1} u\|_{L^\infty(\mathbb{R}^3)}
   \les
   \|\Delta_{-1} u\|_{L^2(\mathbb{R}^3)}
   \les
   \|u\|_{L^2(\mathbb{R}^3)}
   .
   \llabel{EQ09}
  \end{equation}
For the high-frequency part, we get
  \begin{align}
   \begin{split}
   \|(\id - \Delta_{-1}) Du\|_{L^\infty(\mathbb{R}^3)}
   &\les
   \sum_{j > -1} 2^{-j} \|\Delta_j D^2 u\|_{L^\infty(\mathbb{R}^3)}
   \les
   \sum_{j > -1} 2^{-j(1-3/p)} \|\Delta_j D^2 u\|_{L^p(\mathbb{R}^3)}
   \\&\les
   \frac{1}{1- 2^{-(1-3/p)}} \|D^2 u\|_{L^p(\mathbb{R}^3)}
   .
   \end{split}
   \llabel{EQ10}
  \end{align}
This implies
  \begin{equation}
   \|Du\|_{L^\infty(\mathbb{R}^3)} 
   \les
   \|u\|_{L^2(\mathbb{R}^3)}
   + \|D^2 u\|_{L^p(\mathbb{R}^3)}
   ,
   \llabel{EQ11}
  \end{equation}
from which, after optimizing in the scaling, we obtain
  \begin{equation}
   \|Du\|_{L^\infty(\mathbb{R}^3)}
   \leq
   C\|u\|_{L^2(\mathbb{R}^3)}^{1-\alpha} \|D^2 u\|_{L^p(\mathbb{R}^3)}^{\alpha}
   ,
   \quad\quad
   \text{where } \alpha = \frac{5}{7-6/p}
   .
   \label{EQ12}
  \end{equation}
Now, let $u$ be as assumed in the hypothesis of Theorem~\ref{T01}. 
We define $\tilde \, \colon W^{2,p}(\Omega) \to W^{2,p}(\mathbb{R}^3)$ as the continuous extension operator given in Section~5.4 of~\cite{Ev}, the norm of which is~$O_p(1)$. 
Indeed, the construction only involves a linear combination of several dilated reflections across the hyperplane $\{x_3 = 0\}$ in the case that $\Omega = \mathbb{R}_+^3$.
Estimation of these reflected terms only produces a constant $C = O_p(1)$. 
A simple localization argument gives us the result for the smooth bounded domain. 
Moreover, this operator is continuous from $L^2(\Omega) \to L^2(\mathbb{R}^3)$.

Using~\eqref{EQ12}, we get 
  \begin{equation}
   \|Du\|_{L^\infty}
   \les 
   \|D\tilde{u}\|_{L^\infty (\mathbb{R}^3)}
   \les
   \|\tilde{u}\|_{L^2(\mathbb{R}^3)}^{1-\alpha} \|D^2 \tilde{u}\|_{L^p(\mathbb{R}^3)}^\alpha
   \les
   \|u\|_{L^2}^{1-\alpha} \|u\|_{W^{2,p}}^\alpha
   .
   \label{EQ13}
  \end{equation}
However, the same reasoning as in~\eqref{EQ13} yields
  \begin{align}
   \begin{split}
   \|u\|_{W^{2,p}} 
   &\les
   \|u\|_{L^p}
   + \|Du\|_{L^p}
   + \|D^2 u\|_{L^p}
   \\&\les
   \|u\|_{L^2}^{1-\theta_1}\|u\|_{W^{2,p}}^{\theta_1}
   + \|u\|_{L^2}^{1-\theta_2} \|u\|_{W^{2,p}}^{\theta_2}
   + \|D^2 u\|_{L^p}
   ,
   \end{split}
   \llabel{EQ14}
  \end{align}
where $\theta_1 = (6-3p)/(6-7p)$ and $\theta_2 = (6-5p)/(6-7p)$.
Young's inequality then gives 
  \begin{equation}
   \|u\|_{W^{2,p}} 
   \les 
   \|u\|_{L^2} 
   + \|D^2 u\|_{L^p}
   .
   \llabel{EQ15}
  \end{equation}
Substituting this into~\eqref{EQ13} completes the proof of~\eqref{EQ07}.

Next, we prove~\eqref{EQ08}. 
The proof for a fixed $p$ is contained in~\cite{BB}, but we explicitly track the dependence
of constants on~$p$.
Denote 
  \begin{equation}
   \varphi_{ij} 
   = 
   \partial_j u_i - \partial_i u_j
   .
   \label{EQ16}
  \end{equation}
Note that $\|\varphi \|_{W^{k,p}} \les \|\omega\|_{W^{k,p}}$. 
Applying $\partial_i $ to~\eqref{EQ16} and noting that $\div u = 0$, we get
  \begin{equation}
   -\partial_{ii} u_j 
   = 
   \partial_i \varphi_{ij}
   .
   \label{EQ17}
  \end{equation}
Now, let $\nu \in C^\infty(\overline{\Omega})$ be a smooth extension of the outward normal vector $n$, and denote $U = \nu_k u_k$. 
By direct computation, we obtain 
  \begin{equation}
   \partial_{ii} U
   = 
   -\nu_k \partial_i \varphi_{ij}
   + 2 \partial_i \nu_k \partial_i u_k
   + \partial_{ii} \nu_k u_k
   ,
   \llabel{EQ18}
  \end{equation}
along with the boundary condition, $U|_{\partial \Omega} = 0$.
Elliptic regularity then yields 
  \begin{equation}
   \|U\|_{W^{2,p}}
   \leq 
   C p \|\Delta U\|_{L^p}
   \leq C p \left( \|\varphi\|_{W^{1,p}} + \|u\|_{W^{1,p}}\right)
   .
   \label{EQ19}
  \end{equation}
To see the apparent dependence on $p$ in~\eqref{EQ19}, we refer to Sections 9.4--9.5 in~\cite{GT}. 
In short, for the problem $-\Delta U = f$ with $U|_{\partial \Omega} = 0$, we can represent $U$ using the Newtonian potential of $f$,
  \begin{equation}
   U
   =
   \int_{\Omega} \Gamma(x-y) f(y)\, dy
   ,
   \label{EQ19a}
  \end{equation}
where $\Gamma$ is the fundamental solution to Laplace's equation. 
The convolution is of Calder\'on-Zygmund type, for which we may  derive weak-(1,1) and $L^2$ estimates. 
Interpolating, we may derive an estimate $\|D^2 u\|_{L^p} \leq C_p \|f\|_{L^p}$ for $p \in (1,2)$ where 
  \begin{equation}
   C_p 
   \sim
   \left( \frac{p}{p-1} + \frac{p}{2-p}\right)^{1/p}
   .
   \llabel{EQ19b}
  \end{equation}
However, using duality, we then get an estimate for $p \in (2,\infty)$ with
  \begin{equation}
   C_p 
   \sim
   \left( p + \frac{p}{p-2}\right)^{1-1/p}
   \les
   \,
   p
   \quad\quad
   \text{for }
   p > 4
   .
   \llabel{EQ19c}
  \end{equation}

Next, we note that 
  \begin{equation}
   \nu_k \partial_k u_j 
   = 
   \partial_j U
   - \partial_j \nu_k u_k
   + \nu_k \varphi_{jk}
   ,
   \label{EQ20}
  \end{equation}
so that, by the trace inequality,
  \begin{equation}
   \|n_k \partial_j u_j \|_{W^{1-1/p,p}(\partial \Omega)}
   \les
   \|U\|_{W^{2,p}}
   + \|\varphi\|_{W^{1,p}}
   + \|u\|_{W^{1,p}}
   .
   \label{EQ21}
  \end{equation}
To see this, the proof of~\cite[Theorem~15.20]{L} gives the explicit dependence $C= O_p(1)$ in the case that $\Omega = \mathbb{R}_+^3$. 

By elliptic regularity for the Neumann problem given in~\eqref{EQ17} with~\eqref{EQ20} restricted to the boundary, 
we have
  \begin{equation}
   \|\nabla u_j\|_{W^{1,p}}
   \leq 
   Cp\left( \|\Delta u\|_{L^p} + \|n_k \partial_j u_j\|_{W^{1-1/p,p}(\partial \Omega)}\right)
   .
   \label{EQ22}
  \end{equation}
For the dependence on $p$, we refer to Section~6.7 in~\cite{GT}, from which, in the case that $\Omega = \mathbb{R}_+^3$ and $\partial_3 u_j = 0$, we can represent $u_j$ in terms of the Green's function
  \begin{equation}
   w(x)
   =
   \Gamma(x-y) - \Gamma(x-y^\ast)
   ,
   \llabel{EQ22a}
  \end{equation}
where $y^\ast$ denotes the reflection of $y$ across the hyperplane $\{x_3 = 0\}$. 
We then repeat the analysis outlined for $U$ starting at~\eqref{EQ19a} to get linearity in $p$ for the constant. 
Localization and a change of coordinates complete the argument for a smooth bounded domain~$\Omega$. 
For nonhomogeneous Neumann data, the argument is modified through an appropriate lifting, see~\cite[Theorem~6.26]{GT}.

Upon combining~\eqref{EQ17},~\eqref{EQ19},~\eqref{EQ21}, and~\eqref{EQ22}, we arrive at 
  \begin{equation}
    \|u\|_{W^{2,p}}
    \leq
    Cp^2\left(\|\omega\|_{W^{1,p}} + \|u\|_{W^{1,p}}\right)
    \leq
    Cp^2 \left( \|D\omega\|_{L^p} + \|u\|_{W^{1,p}}\right)
    .
   \llabel{EQ23}
  \end{equation}
Interpolating as in~\eqref{EQ13} and using Young's inequality,
we obtain
  \begin{equation}
   \|u\|_{W^{2,p}}
   \leq
   Cp^2 \|D\omega\|_{L^p} + Cp^{2\gamma} \|u\|_{L^2}
   ,
   \quad\quad
   \text{where }
   \gamma = \frac{7p-3}{2p+3}
   .
   \llabel{EQ24}
  \end{equation}
This completes the proof of the lemma. 
\end{proof}

We now prove the first main theorem.  
\begin{proof}[Proof of Theorem~\ref{T01}]
We claim that for $p > 4$,
  \begin{equation}
   \frac{d}{dt} \|D\omega(t)\|_{L^p} 
   \leq
   C p^{2(1+\alpha)}\left(p^{2(\gamma-1)} + \|D\omega\|_{L^p}\right)^{1+\alpha}
   ,
   \quad\quad
   \text{where }
   \alpha = \frac{5}{7 - 6/p}
   .
   \label{EQ26}
  \end{equation}
Indeed, testing the vorticity equation,
  \begin{equation}
   \omega_t
   + u \cdot \nabla \omega
   =
   \omega \cdot \nabla u
   ,
   \label{EQ26a}
  \end{equation}
with $D\omega_i |D\omega|^{p-2}$, we get
  \begin{equation}
   \frac{1}{p} \frac{d}{dt} \|D\omega\|_{L^p}^p 
   \les
   \|Du\|_{L^\infty} \|D\omega\|_{L^p}^p
   + \|\omega\|_{L^\infty} \|D^2 u\|_{L^p} \|D\omega\|_{L^p}^{p-1}
   .
   \llabel{EQ27}
  \end{equation}
Now, we divide both sides by $\|D\omega\|_{L^p}^{p-1}$ and use Lemma~\ref{L04} to acquire 
  \begin{align}
   \begin{split}
   \frac{d}{dt} \|D\omega\|_{L^p}
   &\les
   p^{2+2\alpha} \|D\omega\|_{L^p}^{1+\alpha}
   + p^{2+2\gamma \alpha} \|D\omega\|_{L^p}
   + p^{2\gamma + 2\alpha} \|D\omega\|_{L^p}^\alpha
   + p^{2\gamma(1+\alpha)} 
   \\&\les
   p^{2(1+\alpha)} \|D\omega\|_{L^p}^{1+\alpha}
   + p^{2\gamma(1+ \alpha)}
   \les
   p^{2(1+\alpha)} \left(p^{2(\gamma-1)} + \|D\omega\|_{L^p}\right)^{1+\alpha}
   ,
   \end{split}
   \llabel{EQ28}
  \end{align}
where we used the conservation of energy in the first inequality and Young's inequality in the second one.
This completes the proof of~\eqref{EQ26}.

Since $\Omega$ is bounded, we have
  \begin{equation}
   \|D\omega\|_{L^p} 
   \leq 
   C\|D\omega\|_{L^\infty}
   \label{EQ25}
  \end{equation}
for all $p \geq 1$. 
With~\eqref{EQ25}, a similar argument as in~\cite[Theorem~2.5]{IK} yields 
  \begin{equation}
   \|D\omega(t)\|_{L^\infty} 
   \geq 
   \left( \frac{1}{C p^{2(1+\alpha)} (T_\ast -t)}\right)^{1/\alpha}
   -p^{2(\gamma-1)}
   .
   \label{EQ29}
  \end{equation}
Since the left-hand side of~\eqref{EQ29} is independent of $p$, we may optimize the right-hand side in~$p$. 
Note that 
  \begin{equation}
   \frac{2(1+\alpha)}{\alpha} 
   = 
   \frac{24}{5} + O\left( \frac 1p\right)
   ,
   \quad\quad
   \text{and}
   \quad\quad
   (T_\ast -t)^{1/\alpha} 
   = 
   (T_\ast -t)^{-6/5p} (T_\ast -t)^{7/5}
   ,
   \llabel{EQ30}
  \end{equation}
where $p^{-1/p}$ is bounded uniformly below.
Also, $C^{-1/\alpha}$ may be bounded below by a constant which we denote the same.
Hence, for $p > 4$, we have
  \begin{equation}
   \|D\omega(t)\|_{L^\infty}
   \geq 
   \frac{1}{C p^{24/5} (T_\ast -t)^{-6/5p} (T_\ast -t)^{7/5}}
   -p^{2(\gamma-1)}
   ,
   \quad\quad
   t \in [0,T_\ast)
   .
   \llabel{EQ31}
  \end{equation}
To optimize in $p$, we define
  \begin{equation}
   F(p,t)
   =
   p^{-24/5} e^{-6L(t)/5p}
   \quad\quad
   \text{and}
   \quad\quad
   L(t) = \log \left( \frac{1}{T_\ast -t}\right)
   .
   \llabel{EQ32}
  \end{equation}
Maximizing $F(\cdot, t)$ for each $t$, we get $p \sim L(t)$. 
Since $2(\gamma-1) = 5 + O(1/p)$ as $p\to\infty$, we have 
  \begin{equation}
   \|D\omega(t)\|_{L^\infty}
   \geq 
   \frac{1}{C (T_\ast -t)^{7/5} \log^{24/5}\left( \frac{1}{T_\ast -t}\right)}
   - \log^5\left( \frac{1}{T_\ast -t}\right)
   ,
   \quad\quad
   T_\ast -t < \frac{1}{C}
   ,
   \label{EQ32a}
  \end{equation}
where $C$ depends on~$\|u_0\|_{L^2}$.
Finally, we consider $t$ sufficiently close to $T_\ast$ such that the logarithm on the far-right of~\eqref{EQ32a} is negligible. 
This completes the proof. 
\end{proof}

Next, we prove Theorem~\ref{T02}.
We start with an interpolation lemma. 
For the rest of the paper, the notation ``$C$'' or ``$\les$'' allows constants that depend on~$p$.
\begin{Lemma}
\label{L05}
Under the same hypothesis as in Theorem~\ref{T02}, we have 
  \begin{equation}
   \|D^{k-1} \omega\|_{L^{2(k+1)}}
   \les
   \|u\|_{L^2}^{1/(k+1)} \|D^k\omega\|_{L^\infty}^{k/(k+1)}
   + \|u\|_{L^2}
   \label{EQ33}
  \end{equation}
for $k \geq 2$.
\end{Lemma}
\begin{proof}
Introduce the exponents 
  \begin{equation}
   p_j 
   =
   \frac{2(k+1)}{k-j}
   ,
   \quad\quad
   j = 0,1,\dots, k-1
   ,
   \llabel{EQ34}
  \end{equation}
and $p_k = \infty$.
By the Gagliardo-Nirenberg interpolation inequality, we have 
  \begin{equation}
   \|\omega\|_{W^{j,p_j}} 
   \les
   \|\omega\|_{W^{j-1,p_{j-1}}}^{1/2} \|\omega\|_{W^{j+1,p_{j+1}}}^{1/2}
   ,
   \quad\quad
   j = 1,\dots, k-1
   .
   \label{EQ35}
  \end{equation}
Furthermore, 
  \begin{equation}
   \|\omega\|_{L^{p_0}}
   \leq 
   \|Du\|_{L^{p_0}}
   \les 
   \|u\|_{L^2}^{1/2} \|D^2 u\|_{L^{p_1}}^{1/2}
   \les
   \|u\|_{L^2}^{1/2} \|\omega\|_{W^{1,p_1}}^{1/2}
   + \|u\|_{L^2}
   ,
   \label{EQ36}
  \end{equation}
where, 
in the last inequality, 
we used Lemma~\ref{L04}.
Substituting~\eqref{EQ36} into~\eqref{EQ35} for $j=1$, we obtain
  \begin{equation}
   \|\omega\|_{W^{1,p_1}}
   \les
   \|u\|_{L^2}^{1/4} \|\omega\|_{W^{1,p_1}}^{1/4} \|\omega\|_{W^{2,p_2}}^{1/2}
   + \|u\|_{L^2}^{1/2} \|\omega\|_{W^{2,p_2}}^{1/2}
   ,
   \llabel{EQ37}
  \end{equation}
which simplifies to
  \begin{equation}
   \|\omega\|_{W^{1,p_1}}
   \les 
   \|u\|_{L^2}^{1/3}\|\omega\|_{W^{2,p_2}}^{2/3} 
   + \|u\|_{L^2}
   .
   \llabel{EQ38}
  \end{equation}
Repeating this process for $j=2,\dots, k$, we acquire
  \begin{equation}
   \|\omega\|_{W^{k-1,p_{k-1}}} 
   \les 
   \|u\|_{L^2}^{1/(k+1)} \|\omega\|_{W^{k,\infty}}^{k/(k+1)}
   + \|u\|_{L^2}
   ,
   \llabel{EQ39}
  \end{equation}
from which interpolation gives us~\eqref{EQ33}, completing the proof of the lemma.
\end{proof}

We now prove Theorem~\ref{T02}.
We only provide a sketch since the argument is similar to that of~\cite[Theorem~3]{IK}.
\begin{proof}[Proof sketch of Theorem~\ref{T02}]
From the vorticity equation~\eqref{EQ26a}, we obtain the estimate 
  \begin{equation}
   \frac{d}{dt} \|D^{k-1} \omega\|_{L^p}
   \les
   \|Du\|_{L^\infty} \|D^{k-1}\omega\|_{L^p}
   + \|\omega\|_{L^\infty} \|D^ku\|_{L^p}
   \llabel{EQ40}
  \end{equation}
for any $p \geq 2$, though we fix $p = 2(k+1)$ later.
Using similar reasoning as in Lemma~\ref{L04}, we may bound $\|D^k u\|_{L^p}$ and $\|Du\|_{L^\infty}$, though we are not concerned with the $p$-dependence in the constants. 
This gives us 
  \begin{equation}
   \frac{d}{dt} \|D^{k-1}\omega\|_{L^p}
   \les 
   \left(1 + \|D^{k-1}\omega\|_{L^p}\right)^{1+\alpha}
   ,
   \quad\quad
   \text{where }
   \alpha = \frac{5}{2k+3 - 6/p}
   .
   \llabel{EQ41}
  \end{equation}
In combination with Lemma~\ref{L05} and conservation of energy, we have 
  \begin{equation}
   \|D^{k}\omega\|_{L^\infty}
   \geq 
   \frac{1}{C (T_\ast -t)^{2k/5 + 1}}
   -1
   ,
   \quad\quad
   t \in [0,T_\ast)
   ,
   \label{EQ42}
  \end{equation}
after noting $((k+1)/k)(1/\alpha) = 2k/5 +1$.
By considering $T_\ast -t < 1/C$ where $C$ depends on $\|u_0\|_{L^2}$, we may neglect the constant on the right-hand side of~\eqref{EQ42}. 
\end{proof}

\section{Proof of Theorem~\ref{T03}}\label{sec03}
The proof begins with verifying that a solution to~\eqref{EQ43} below is a solution to~\eqref{EQ05}. 
We then construct the claimed solution using~\eqref{EQ43}. 

\begin{proof}[Proof of Theorem~\ref{T03}]
Let $X \colon [0,\infty) \to (0,\infty)$ be a positive $C^1$ function satisfying 
  \begin{equation}
   \frac{d}{d\tau} X(\tau) 
   \leq
   1 + X(\tau)e^\tau
   ,\quad\quad
   \tau \in [0,\infty)
   ,
   \label{EQ43}
  \end{equation}
with $X(0) = M$, $\limsup_{\tau \to \infty} X(\tau) = \infty$, and $X(\tau_n) = M/(n+1)$ along some sequence $\tau_n \to \infty$.
Also, we assume 
  \begin{equation}
   T_\ast 
   := 
   \int_0^\infty \frac{1}{X(\tau)} \, d\tau 
   < \infty
   .
   \label{EQ44}
  \end{equation}
Next, we define 
  \begin{equation}
   t(\tau) 
   = 
   \int_0^\tau \frac{1}{X(\tau')} \, d\tau'
   .
   \llabel{EQ45}
  \end{equation}
Since $X$ is positive, it follows that $t \colon [0,\infty) \to [0,T_\ast)$ is a $C^1$ diffeomorphism.
We denote its inverse as $\tau \colon [0,T_\ast) \to [0,\infty)$. 

We claim that $x(\cdot) := X(\tau(\cdot))$ satisfies the properties (i)--(iv) of the theorem. 
Indeed, (i) and (iii) immediately hold since $t$ and $\tau$ increase with respect to each other and $\tau(0) = 0$. 
For (ii), differentiating $\tau(t(\tau)) = \tau$ gives us $(d\tau/dt)(t(\tau)) = X(\tau)$.
Evaluating at $\tau(t)$ yields 
  \begin{equation}
   \frac{d\tau}{dt}(t)
   = 
   x(t)
   .
   \llabel{EQ46}
  \end{equation}
Hence,
  \begin{equation}
   \tau(t) 
   = 
   \int_0^t x(s)\, ds
   \llabel{EQ47}
   .
  \end{equation}
Moreover, from~\eqref{EQ43}, we have
  \begin{equation}
   \frac{d}{dt} x(t)
   =
   \frac{d}{dt} X(\tau(t))
   =
   \frac{dX}{d\tau}(\tau(t)) \frac{d\tau}{dt}(t)
   \leq
   x(t) + x(t)^2 \exp\left(\int_0^t x(s)\, ds\right)
   .
   \llabel{EQ48}
  \end{equation}
  Finally, for (iv), let $\tau_n \to \infty$ be a given sequence such that $X(\tau_n) = M/(n+1)$.
Then, by invertibility, there exists a sequence $t_n \to T_\ast^-$ such that $\tau_n = \tau(t_n)$ and $x(t_n) = \frac{M}{n+1}$.

It is now sufficient to construct a function $X\colon [0,\infty) \to (0,\infty)$ satisfying the properties outlined in~\eqref{EQ43}--\eqref{EQ44}. 
We denote $\xi_n = 2^{-n-1}$ and define
  \begin{equation}
   Y(\tau) 
   = 
   \begin{cases} 
    e^{\tau/2} - e^{n/2} + \frac{1}{n+1}
    ,
    & 
    \tau \in [n, n+1 - \xi_n)
    ,
    \\
    \phi_n(\tau)
    ,
    &
    \tau \in [n+1 - \xi_n, n+1)
    ,
   \end{cases}
   \label{EQ49}
  \end{equation}
where $\phi_n$ is a $C^1$ extension of the first function in~\eqref{EQ49} such that $Y \in C^1([0,\infty))$ with
  \begin{equation}
   \phi_n''(\tau) < 0 
   \inon{for $\tau \in [n+1-\xi_n, n+1- \xi_n/2) =: I_n$}
   \llabel{EQ50}
  \end{equation}
and 
  \begin{equation}
   \phi_n''(\tau) \geq 0
   \inon{for $\tau \in [n+1-\xi_n/2, n+1) =: J_n$}
   .
   \llabel{EQ51}
  \end{equation}
We also assume that $\phi_n(\tau) \geq 1/(n+1)$ for each $n \geq 0$.

We claim that $X(\cdot) := MY(\cdot)$ has the desired profile. 
Indeed, we immediately verify that $X(0) = M$, $\limsup_{\tau \to \infty} X(\tau) = \infty$, and $X(n) = M/(n+1)$ for all $n \geq 0$. 
Furthermore, $X$ satisfies the inequality~\eqref{EQ43} on $[0,\infty)$. 
To see this, we first show that, for any $n \geq 0$, 
  \begin{equation}
   \frac{M}{2} e^{\tau/2} 
   \leq
   1 + M \left(e^{\tau/2} - e^{n/2} + \frac{1}{n+1}\right)e^\tau
   ,
   \inon{$\tau \in [n,n+1-\xi_n)$}
   .
   \llabel{EQ52}
  \end{equation}
Taking a difference on this interval, we obtain
  \begin{equation}
    1 
    + M\left( e^{\tau/2} - e^{n/2} + \frac{1}{n+1}\right) e^\tau 
    - \frac{M}{2} e^{\tau/2}
    \geq 
    1 
    +\frac{M}{n+1} e^\tau
    -\frac{M}{2} e^{\tau/2}
    \geq 
    1
    ,
    \llabel{EQ53}
  \end{equation}
where we used $\tau \geq n$ and the bound $2e^x \geq 1 + 2x$ for $x \geq 0$, with $x = n/2$. 

For $M\phi_n$, we deal with $I_n$ and $J_n$ separately.
We split $I_n$ into two consecutive intervals $I_n = I_{n,1} \cup I_{n,2}$, on which $\phi_n' > 0$ and $\phi_n' \leq 0$, respectively. 
Thus, $\phi_n'$ is decreasing on $I_{n,1}$ while $\phi_n$ is increasing. 
By continuity, $\phi_n$ satisfies~\eqref{EQ43} at $\tau = n+1 - \xi_n$. 
Hence, the inequality persists on $I_{n,1}$
if $\phi_n$ is chosen properly.
For $I_{n,2}$,~\eqref{EQ43} immediately verifies since the right-hand side is positive. 
On $J_n$, the same reasoning holds since $\phi_n'' > 0$ and $M\phi_n(t) \geq M/(n+1)$. 
The inequality~\eqref{EQ43} is now established for  $\tau \in [0,\infty)$.

Finally, by direct computation, we have
  \begin{equation}
   \int_0^\infty \frac{1}{Y(\tau)} \, d\tau
   \les
   \sum_{n=0}^\infty (n+1)(e^{-n/2} + \xi_n)
   =: \Gamma < \infty
   .
   \llabel{EQ54}
  \end{equation}
Hence,
  \begin{equation}
   T_\ast
   :=
   \int_0^\infty \frac{d\tau}{X(\tau)}
   = M^{-1} \Gamma
   ,
   \llabel{EQ55}
  \end{equation}
completing the proof. 
\end{proof}

\section*{Acknowledgments}
The authors were supported in part by the NSF grant DMS-2205493.

\end{document}